\documentclass[11pt]{amsart}
\usepackage{amscd,amsfonts,amssymb,mathtools,mathrsfs,xcolor}
\usepackage[all]{xy}
\newtheorem{theorem}{THEOREM}[section]
 \newtheorem{corollary}[theorem]{COROLLARY}
 \newtheorem{lemma}[theorem]{LEMMA}
 \newtheorem{proposition}[theorem]{PROPOSITION}
 \newtheorem{conjecture}[theorem]{CONJECTURE}

 \newtheorem{definition}[theorem]{\emph{Definition}}

 \newcommand{\iso}{\approx}

\newcommand{\ind}{\text{\rm Ind }}
\newcommand{\res}{\text{\rm Res }}
\newcommand{\inff}{\text{\rm Inf }}
\newcommand{\deff}{\text{\rm Def }}

\usepackage{newtxtext,newtxmath}
\begin{document}

\title[Relative Brauer Relations]{Relative Brauer Relations of Abelian p-groups}%
\author{Marian F. Anton}%
\address{Department of Mathematics, CCSU, New Britain, CT 06050, U.S.A. and I.M.A.R., P.O. Box 1-764, Bucharest, RO 014700}%
\email{anton@ccsu.edu}%

\thanks{}%

\date{\today}%
\begin{abstract}
The Brauer relations of a finite group $G$ are virtual differences of non-isomorphic $G$-sets $X-Y$ which induce isomorphic permutation $G$-representations $\mathbb Q[X]\simeq\mathbb Q[Y]$ over the rationals. These relations have been classified by Tornehave-Bouc and Bartel-Dokchitser.  Motivated by stable homotopy theory, a relative version of Brauer relations for $(G,C_p)$-bisets which are $C_p$-free have been classified by Kahn in case $G$ is an elementary Abelian $p$-group. In this paper we extend Kahn's classification to the case when $G$ is a finite Abelian $p$-group.  
\end{abstract}
\maketitle

\section{Introduction}

The Burnside ring $B(G)$ of a finite group $G$ is the Grothendick ring of the category of finite $G$-sets and is isomorphic up to completion to the stable cohomotopy group of the classifying space $B_G$ according to Segal's Conjecture \cite{Car84}. The relative Burnside module $B(G,H)$ of a pair of finite groups $(G,H)$ is the Grothendick module of the category of finite $(G,H)$-bisets that are $H$-\emph{free}. Up to completion, $B(G,H)$ describes the stable homotopy classes of maps from the classifying space $B_G$ to the classifying space $B_H$,  by the generalized Segal's Conjecture \cite{May85}. 

The representation ring $R_F(G)$ of a finite group $G$ over a field $F$ is the Grothendick ring of the category of finitely generated $F G$-modules, where $FG$ denotes the group ring of $G$ over $F$. Let $F=\mathbb Q$ be the field of rational numbers. The functor sending each finite $G$-set $X$ to the permutation $G$-module $\mathbb Q[X]$ induces a ring homomorphism from the Burnside ring $B(G)$ to the rational representation ring $R_\mathbb Q (G)$:
\begin{align}
f_G:B(G)\to R_\mathbb Q(G),&& f_G[X]=\mathbb Q[X].
\end{align}
The cokernel of $f_G$ is finite of exponent dividing the group order $|G|$ by Artin's induction theorem \cite{Ser78} and it is trivial if $G$ is a finite $p$-group by Ritter-Segal \cite{Rit72,Seg72}. The elements of the kernel $K(G)$ of $f_G$ are called the \emph{ Brauer relations} of the group $G$ or the \emph{$G$-relations} and these have been classified by  Tornehave-Bouc \cite{Bou04} for finite $p$-groups and Bartel-Dokchitser \cite{Bar15} for arbitrary finite groups. 

We note that the image of the submodule $B(G,H)\subset B(G\times H)$ under the map $f_{G\times H}$ is contained in the Grothendieck submodule $R_\mathbb Q(G,H)\subset R_\mathbb Q(G\times H)$ of the category of finitely generated $\mathbb Q(G\times H)$-modules which are \emph{free right} $\mathbb QH$-modules.  The kernel and cokernel of the well defined restricted homomorphism 
\begin{align}\label{map}
f_{G,H}=f_{G\times H}:B(G,H)\xrightarrow{} R_\mathbb Q(G, H)
\end{align}
are also of interest in view of the generalized Segal's Conjecture. In particular, it makes sense to call the elements of the kernel $K(G,H)$ of $f_{G,H}$ the \emph{relative Brauer relations} of the pair $(G,H)$ or the $(G,H)${\em-relations}. The cokernel of $f_{G,H}$ is trivial for $G$ a finite $p$-group and $H=C_p$ by Anton \cite{Ant06} and the $(G,C_p)$-relations have been classified by Kahn \cite{Kah13} for $G$ an elementary Abelian $p$-group, where $C_p$ \emph{denotes the cyclic group of prime order $p$}. The main result of this paper is

\begin{theorem}\label{mainb} The relative Brauer relations of the pair $(G,C_p)$ for $G$ a finite Abelian $p$-group are linear combinations of relative Brauer relations `indufted' from sub-quotients of $G\times C_p$ of the form $C_p\times C_p\times C_p$. 
\end{theorem}

We note that there is a natural ring homomorphism mapping $R_F(G)$ to the stable homotopy classes of maps from the classifying space $B_G$ to the plus construction of the classifying space $B_{GL(F)}$, where $GL(F)$ is the infinite general linear group over the field $F$. Up to completion, this map is an isomorphism for $F$ the (topological) field of complex numbers by Atiyah's Theorem \cite{Ati69} or for $F$ a finite field \cite{Kass89}. If $F=\mathbb Q$, this homomorphism connects Brauer relations with algebraic $K$-theory  \cite{Mit92}. The relative Brauer relations are connected with maps between algebraic $K$-theory spectra.

The background terminology will be reviewed in section \S 2 and a precise formulation of Theorem \ref{mainb} will be given in section \S 3. In section \S 4 we prove some key rank lemmas. In section \S 5 we prove the main theorem up to torsion. In section \S 6 we reduce the proof to a set of special generators and finish the argument in \S 7. 

\section{Background and terminology}

This section is a survey of basic definitions and facts about Burnside and representation modules, many of them being used in this paper. 

\subsection{Relative Burnside modules} Following \cite{Bar15}, the \emph{Burnside ring} $B(\Gamma)$ of a group $\Gamma$ is the free Abelian group generated by the isomorphism classes $[X]$ of finite $\Gamma$-sets $X$ modulo the relations $[X\sqcup Y]=[X]+[Y]$ where $\sqcup$ denotes the disjoint union. The product in $B(\Gamma)$ is given by $[X]*[Y]=[X\times Y]$ where $X\times Y$ is the $\Gamma$-set under the diagonal $\Gamma$-action. 

\begin{proposition}[\cite{Bar15}] \label{bbasis} The transitive $\Gamma$-sets are \emph{left coset spaces} $\Gamma/L$ of subgroups $L\le \Gamma$ and their isomorphism classes $[\Gamma/L]$ form a basis for $B(\Gamma)$.  
\end{proposition}

There is a bijection sending the conjugacy class of a subgroup $L\le \Gamma$ to  the  basis element  $[\Gamma/L]$. Using this identification,\\

\emph{We write the elements of $B(\Gamma)$ as integral linear combinations $\sum n_iL_i$ of subgroups $L_i\le \Gamma$ up to conjugacy.} \\

The product of two basis elements in $B(\Gamma)$ is given by the double coset formula
\begin{align}
L*M=\sum_{x\in L\backslash \Gamma/M} L\cap \prescript{x}{}M
\end{align}
where $L,M\le \Gamma$ and $\prescript{x}{}M=xMx^{-1}$ for $x\in\Gamma$.

Given a pair $(G,H)$ of finite groups, let $\Gamma=G\times H$. A $(G,H)$-\emph{biset} is a \emph{finite} set $X$, endowed with a left $G$-action and a right $H$-action that commute with each other, or equivalently, a left $\Gamma$-action where the right $H$-action is defined via  
\begin{align}\label{right}
xh=h^{-1}x\text{ for } x\in X\text{ and }h\in H.
\end{align}
\begin{definition}The \emph{relative Burnside module} $B(G,H)$ of a pair of finite groups $(G,H)$ is the free Abelian group generated by the isomorphism classes $[X]$ of $(G,H)$-bisets which are \emph{right  $H$-free}, modulo the relations $[X\sqcup Y]=[X]+[Y]$.\end{definition} 

According to \cite{May85}, the transitive right $H$-free $(G,H)$-bisets are {\em twisted products} $G\times_\rho H$ between $G$ and a group homomorphism $\rho:K\to H$ from a subgroup $K\le G$. Such a product is the quotient of $\Gamma$ modulo the relations 
$$
[gk,h]=[g,\rho(k)h]\text{ for }g\in G, h\in H,\text{ and }k\in K.
$$ 
Here $[g,h]$ denotes the class of $(g,h)$ in $G\times_\rho H$. The map $\Gamma\to G\times H$ given by $(g,h)\mapsto (g,h^{-1})$ induces an isomorphism of left $\Gamma$-sets $\Gamma/(K\times\rho)\iso G\times_\rho H$ where 
\begin{align}K\times\rho=\{(k,\rho(k)):k\in K\} \label{grp}\end{align}
is \emph{the graph (subgroup) of} $\rho$ in $\Gamma$. This is an isomorphism of $(G,H)$-bisets via \eqref{right}. 

\begin{proposition} [\cite{May85}]\label{rbbasis} 
A basis for the submodule $B(G,H)\subset B(\Gamma)$ consists of the isomorphism classes of twisted products $[G\times_\rho H]$. 
\end{proposition} 

There is a bijection between these basis elements and the conjugacy classes of group homomorphisms $\rho:K\to H$ with $K\le G$ or equivalently of subgroups $K\times \rho\le \Gamma$. Using this identification, \\

\emph{We write the elements of $B(G,H)$ as integral linear combinations $\sum n_iK_i\times\rho_i$ of graphs up to conjugacy of homomorphisms $\rho_i:K_i\to H$ from subgroups $K_i\le G$.}\\

If $Z$ is a $G$-set and $X$ is a $(G,H)$-biset, then the product $[Z][X]=[Z\times X]$ defines a left $B(G)$-module structure on $B(G,H)$, where $G$ acts on $Z\times X$ diagonally from the left and $H$ acts only on $X$ from the right. The $B(G)$-module structure on $B(G,H)$ is made explicit for $M,K\le G$ and $\rho:K\to H$ by the product:
\begin{align}
M*(K\times\rho)=\sum_{x\in K\backslash G/M}(K\cap \prescript{x}{}M)\times\rho.
\end{align}

\subsection{Functorial operations} These operations are $\mathbb Z$-linear maps on $B(\Gamma)$. The \emph{induction} $\ind_L^\Gamma:B(L)\to B(\Gamma)$ from a subgroup $L\le\Gamma$ is defined  on $L$-sets $Y$ by 
\begin{align*}
\ind_L^\Gamma Y=\Gamma\times_L Y\text{ where }\Gamma\text{ acts by left multiplication on }\Gamma.
\end{align*}
Here $\Gamma\times_LY$ denotes the quotient of $\Gamma\times Y$ modulo the relations $
[\gamma l,y]=[\gamma, ly]$ for $\gamma\in\Gamma, l\in L$ and $y\in Y$. 

The \emph{restriction} $\res_L^\Gamma:B(\Gamma)\to B(L)$ is defined on $\Gamma$-sets $X$ by
\begin{align*}
\res_L^\Gamma X=X\text{ where }L\text{ acts on }X\text{ as a subgroup of }\Gamma.
\end{align*}
On basis elements, for $K\le L$ and $M\le\Gamma$ we have 
\begin{align}\label{mac}
\ind_L^\Gamma K=K, && \res^\Gamma_LM=\sum_{x\in L\backslash\Gamma/M}L\cap\prescript{x}{}M.
\end{align}
The \emph{inflation} $\inff_\Pi^\Gamma:B(\Pi)\to B(\Gamma)$ from a quotient $\Pi=\Gamma/N$ by a normal subgroup $N\trianglelefteq\Gamma$ is defined on $\Pi$-sets $Z$ by
\begin{align*}
\inff_\Pi^\Gamma Z=Z\text{ where } \Gamma\text{ acts on }Z\text{ via its projection in }\Pi.
\end{align*}
The \emph{deflation} $\deff_\Pi^\Gamma:B(\Gamma)\to B(\Pi)$ is given on $\Gamma$-sets $X$ by the \emph{orbit space}  of $N$
\begin{align*}
\deff_\Pi^\Gamma X=N\backslash X\text{ where }[\gamma][x]=[\gamma x]\text{ for }\gamma\in\Gamma \text{ and }x\in X.
\end{align*}
Here $[\gamma]$ denotes the coset $\gamma N$ in $\Pi$ and $[x]$ denotes the \emph{orbit} $Nx$  of $x$ in $N\backslash X$. On basis elements, for $N\trianglelefteq K \le \Gamma$ and $M\le\Gamma$ we have
\begin{align}\label{inf}
\inff_\Pi^\Gamma (K/N)=K, && \deff_\Pi^\Gamma M=NM/N.
\end{align}

\subsection{Relative Representation Modules}

Following \cite{Ser78}, the {\em rational representation ring} $R_\mathbb Q (\Gamma)$ of a finite group $\Gamma$ is the free Abelian group generated by the isomorphism classes $[V]$ of finitely generated left $\mathbb Q\Gamma$-modules $V$ modulo the relations $[V\oplus W]=[V]+[W]$ where $\oplus$ denotes the direct sum.  The product is given by $[V]*[W]=[V\otimes W]$ where $\otimes$ denotes the tensor product over $\mathbb Q$ and $V\otimes W$ is the $\mathbb Q\Gamma$-module under the diagonal $\mathbb Q\Gamma$-action. The irreducible $\mathbb Q\Gamma$-modules are direct summands $V_i$ of the group ring $\mathbb Q\Gamma$ and their isomorphism classes $[V_i]$ form a basis for $R_{\mathbb Q}(\Gamma)$.

Given a pair of finite groups $(G,H)$, let $\Gamma=G\times H$. A $(G,H)${\em-bimodule over the rationals} is simply a finitely generated left $\mathbb Q\Gamma$-module $V$ with $\mathbb QG$ acting on the left via the canonical inclusion in $\mathbb Q\Gamma$ and $\mathbb Q H$ acting on the right via the rule
\begin{align}
vh=h^{-1}v,\text{ for }v\in V\text{ and }h\in H.
\end{align}
\begin{definition}The \emph{relative rational representation module} $R(G,H)$ of  a pair of finite groups $(G,H)$ is the submodule of $R_{\mathbb Q}(\Gamma)$ generated by the isomorphism classes of $(G,H)$-bimodules over the rationals, which are \emph{right $\mathbb QH$-free} modules.\end{definition} 

We call a right $\mathbb QH$-free $(G,H)$-bimodule $V$ over the rationals {\em irreducible} if $V$ cannot be decomposed as a direct sum of right $\mathbb QH$-free $(G,H)$-bimodules over the rationals. In particular,  the isomorphism classes $[W_i]$ of irreducible right $\mathbb Q H$-free $(G,H)$-bimodules $W_i$ over the rationals form a basis for $R(G,H)$. Notice that $\mathbb QH$ with $\mathbb QH$ acting on the right by the multiplication in $H$ and $\mathbb QG$ acting on the left by the identity $1\in G$ is an example of an irreducible right $\mathbb QH$-free $(G,H)$-bimodule over the rationals which is not an irreducible left $\mathbb Q\Gamma$-module (unless $H=1$).  

If $U$ is a left $\mathbb Q G$-module and $V$ a right $\mathbb QH$-free $(G,H)$-bimodule over the rationals, then the product $[U]*[V]=[U\otimes V]$ defines a left $R_{\mathbb Q}(G)$-module structure on $R(G,H)$, where $\mathbb Q G$ acts on $U\otimes V$ diagonally from the left and $\mathbb Q H$ acts only on $V$ from the right. Indeed, with these actions, $U\otimes V$ is a right $\mathbb QH$-free $(G,H)$-bimodule.

From the relative Burnside module $B(G,H)$ to the relative rational representation module $R(G,H)$ we have the natural linear map $f_{G,H}$ defined by \eqref{map}.

\begin{definition}\label{rel} The \emph{relative Brauer relations} of a pair of finite groups $(G,H)$ are the elements of the kernel $K(G,H)$ of the linear map $f_{G,H}:B(G,H)\to R(G,H)$.
\end{definition}

Recall that the Brauer relations of the finite group $\Gamma=G\times H$ are the elements of the kernel $K(\Gamma)$ of the linear map $f_\Gamma:B(\Gamma)\to R_{\mathbb Q}(\Gamma)$ assigning to a $\Gamma$-set $X$ the rational permutation representation $\mathbb Q[X]$. The map $f_{G,H}$ is the restriction of $f_\Gamma$ to the submodule $B(G,H)\subseteq B(\Gamma)$ and its kernel is a submodule $K(G,H)\subseteq K(\Gamma)$.  

\subsection{Functorial operations} 
On the representation ring $R_{\mathbb Q}(\Gamma)$ we define functorial operations, which are $\mathbb Z$-linear maps. The \emph{induction} $\ind_L^\Gamma:R_{\mathbb Q}(L)\to R_{\mathbb Q}(\Gamma)$ from a subgroup $L\le\Gamma$ is defined  on $\mathbb Q L$-modules $W$ by 
\begin{align*}
\ind_L^\Gamma W=\mathbb Q\Gamma\otimes_{\mathbb Q L} W\text{ where }\mathbb Q\Gamma\text{ acts by left multiplication on }\mathbb Q\Gamma.
\end{align*}
The \emph{restriction} $\res_L^\Gamma:R_{\mathbb Q}(\Gamma)\to R_{\mathbb Q}(L)$ is defined on $\mathbb Q\Gamma$-modules $V$ by
\begin{align*}
\res_L^\Gamma V=V\text{ where }\mathbb QL\text{ acts on }V\text{ as a subring of }\mathbb Q\Gamma.
\end{align*}
The \emph{inflation} $\inff_\Pi^\Gamma:R_{\mathbb Q}(\Pi)\to R_{\mathbb Q}(\Gamma)$ from a quotient $\Pi=\Gamma/N$ by a normal subgroup $N\trianglelefteq\Gamma$ is defined on $\mathbb Q\Pi$-modules $U$ by
\begin{align*}
\inff_\Pi^\Gamma U=U\text{ where } \mathbb Q\Gamma\text{ acts on }U\text{ via its projection in }\mathbb Q\Pi.
\end{align*}
The \emph{deflation} $\deff_\Pi^\Gamma:R_{\mathbb Q}(\Gamma)\to R_{\mathbb Q}(\Pi)$ is given on $\mathbb Q\Gamma$-modules $V$ by 
\begin{align*}
\deff_\Pi^\Gamma V=\mathbb Q\Pi\otimes_{\mathbb Q\Gamma} V\text{ where }\mathbb Q\Gamma \text{ acts on }\mathbb Q\Pi\text{ via its projection.}
\end{align*}

\noindent{\em Notation.} For any group $C$ we denote by $1_C$ the trivial rational representation $[\mathbb Q]$ where $C$ acts on $\mathbb Q$ by the identity $1\in C$.  Also $\ind^\Gamma=\ind_C^\Gamma$ and $R(\Gamma)=R_{\mathbb Q}(\Gamma)$. Similarly, $\inff^\Gamma=\inff^\Gamma_\Pi$, $\res^\Gamma=\res^\Gamma_C$, $\deff^\Gamma=\deff^\Gamma_\Pi$ where $\Pi$ is a quotient of $\Gamma$.

\begin{proposition}[\cite{Bar15}]\label{nat}
The linear map $f_\Gamma:B(\Gamma)\to R(\Gamma)$ commutes with the operations $\ind^\Gamma$, $\inff^\Gamma$, $\res^\Gamma$, $\deff^\Gamma$. 
\end{proposition}

For example, if $L\le\Gamma$ is a subgroup and $K$ is a basis element in $B(\Gamma)$ given by the conjugacy class of a subgroup in $\Gamma$, we have Mackey's formula \cite{Ser78} from \eqref{mac}
$$
\res_L^\Gamma(f_\Gamma K)=\res_L^\Gamma\ind^\Gamma(1_K)=\sum_{x\in L\backslash\Gamma/K}\ind^L(1_{L\cap\prescript{x}{}K})= f_L(\res_L^\Gamma(K)).
$$

However, for $\Gamma=G\times H$ these operations restricted to the submodules $B(G,H)$ and $R(G,H)$ do not land in the same submodules except in special cases. For example, if $K\le G$ is a subgroup, and $N\trianglelefteq G$ is a normal subgroup with $\Pi=G/N$, we have a commutative diagram of linear maps
\begin{equation}\label{comm}
\xymatrix{B(K,H)\ar[r]^{\ind_K^G}\ar[d]_{f_{K,H}}&B(G,H)\ar[d]_{f_{G,H}}&
\ar[l]_{\inff_{\Pi}^G}B(\Pi,H)\ar[d]^{f_{\Pi,H}}\\
R(K,H)\ar[r]^{\ind_{K}^G}&R(G,H)&\ar[l]_{\inff_{\Pi}^G}R(\Pi,H)}
\end{equation}
as $\ind_K^G=\ind_{K\times H}^\Gamma$ and $\inff_\Pi^G=\inff_{\Pi\times H}^\Gamma$ do not change the $H$-structure.
\begin{theorem}[Artin's Induction \cite{Ser78}]\label{artin} The vector space $\mathbb Q\otimes R(\Gamma)$ has a basis given by $\ind^\Gamma(1_C)$ where $C$ runs over the conjugacy classes of cyclic subgroups of $\Gamma$.
\end{theorem}

It is a known fact \cite{Ser78} that $\ind^\Gamma(1_C)$ do not generate $R(\Gamma)$ in general. 

\begin{corollary}\label{rank} For $\Gamma=G\times H$, the rank of the $\mathbb Z$-module $R(\Gamma)$ equals  the number of conjugacy classes of cyclic subgroups of $\Gamma$ and the rank of the $\mathbb Z$-module $K(\Gamma)$ equals the number of conjugacy classes of non-cyclic subgroups of $\Gamma$. 
\end{corollary}

\begin{proof}
The conjugacy classes of subgroups $L_i\le\Gamma$ form a basis for $B(\Gamma)$ and their images under $f_\Gamma:B(\Gamma)\to R(\Gamma)$ are given by $f_\Gamma(L_i)=\ind^\Gamma(1_{L_i})$. By Artin's Induction, it follows that in the exact sequence below, $\text{\rm Coker }(f_\Gamma)$ is torsion: 
$$
0\to K(\Gamma)\to B(\Gamma)\xrightarrow{f_\Gamma} R(\Gamma)\to \text{\rm Coker } (f_\Gamma)\to 0.
$$
This concludes the proof as the alternating sum of the ranks is zero.
\end{proof}

\section{The formulation of the main theorem}

We denote by $C_p$ the multiplicative cyclic group of prime order $p$. Since the rank of the $\mathbb Z$-module $K(G)$ of Brauer relations equals the number of non-cyclic subgroups of $G$ by Corollary \ref{rank}, the first part of the following proposition is immediate. 

\begin{proposition} [\cite{Bou04}]\label{teta}
Cyclic groups have no Brauer relations except zero. Products $P=C_p\times C_p$ of cyclic groups of prime order  have one independent Brauer relation given below where $C$ runs over all non-trivial cyclic subgroups of $P$ $$ \Theta_P=1-\sum_{C=\text{ cyclic }}C+pP.$$
\end{proposition}

Recall that an element $\Theta=\sum n_iH_i$ is a $G$-relation if and only if 
\begin{align*}
\sum n_i \text{ Ind}^G(1_{H_i})=0, && 1_{H_i}=\text{ trivial }H_i\text{-module }\mathbb Q.
\end{align*}
According to Proposition \ref{nat}, Brauer relations can be induced by $\ind^{G'}$ from subgroups  $H\le G'$ (or restricted by $\res_H$) and can be lifted by $\inff^{\tilde G}$ from quotients $G=\tilde{G}/N$ (or projected by $\deff_{\tilde{G}/N}$). A Brauer relation of $G$ is called {\em primitive} if cannot be induced from a proper subgroup or lifted from a proper sub-quotient of $G$. 

\begin{theorem}[Bouc-Tornehave \cite{Bou04}]\label{brauer}
All Brauer relations of a $p$-group $G$ are linear combinations of the form
\begin{align*}
\Theta=\sum n_P \text{\rm Ind}^G_K\;\text{\rm  Inf}^{K}_P(\Theta _{P}), && P=K/N\text{ sub-quotient of }G,\; n_P\in\mathbb Z
\end{align*}
of relations $\text{\rm Ind}^G_K\;\text{\rm  Inf}^{K}_P(\Theta _{P})$ {\em `indufted'}  from primitive $P$-relations $\Theta_P$ where 
\begin{enumerate}
\item $P\approx C_p\times C_p$ and  $\Theta_P=1-\sum_{C=\text{ cyclic}}C+pP$ or 
\item $P\approx $ the Heisenberg group of order $p^3$ or  $P\approx $ the dihedral group of order $2^n$ with $n\ge 4$ and
$
\Theta_P=I-IZ-J+JZ.
$
\end{enumerate}
In the second case, $Z$ is the center of $P$ and $I$, $J$ are non-conjugate subgroups of $P$ of order $p$ (or order $2$) intersecting $Z$ trivially. 
\end{theorem}

Let $G$ be a finite $p$-group and $K(G,C_p)$ be the module of relative $(G,C_p)$-Brauer relations as in Definition \ref{rel}. In this case we have the following commutative diagram of short exact sequences of $\mathbb Z$-modules and linear maps
\begin{equation}\label{exactseq}
\begin{CD}
0@>>>K(G\times C_p)@>incl.>>B(G\times C_p)@>f_{G\times C_p}>>R(G\times C_p)@>>>0\\
@.@A incl. AA @A incl. AA @A incl. AA@.\\
0@>>>K(G, C_p)@>incl.>>B(G, C_p)@>f_{G,C_p}>>R(G, C_p)@>>>0
\end{CD}
\end{equation}

The surjectivity of the maps $f_{G\times C_p}$ and $f_{G,C_p}$ has been established in \cite{Seg72} and \cite{Ant06}. By estimating the ranks of the $\mathbb Z$-modules involved in the diagram and searching for Brauer $P$-relations $\Theta'_P$ for sub-quotients $P=K/N$ of the group $G\times C_p$ such that $\ind_K^{G\times C_p}\inff^K_P(\Theta'_P)$ are relative Brauer $(G,C_p)$-relations, we formulate the following   

\begin{conjecture}\label{conj}
Let $p$ be a prime and $G$ a finite $p$-group. 
 All relative Brauer $(G,C_p)$-relations for a $p$-group $G$ are linear combinations 
\begin{align*}
\Theta=\sum n_P\text{\rm Ind}^{G\times C_p}_K\;\text{\rm  Inf}^{K}_P(\Theta' _{P}), && P=K/N\text{ sub-quotient of }G\times C_p,\; n_P\in\mathbb Z
\end{align*}
of $(G,C_p)$-relations $\text{\rm Ind}^{G\times C_p}_K\;\text{\rm  Inf}^{K}_P(\Theta' _{P})$ {\em `indufted'}  from $P$-relations $\Theta'_P$ where 
\begin{enumerate}
\item $P\approx C_p\times C_p\times C_p$ or 
\item $P\approx $ (the Heisenberg group of order $p^3$) $\times\; C_p$ or 
\item $P\approx$ (the dihedral group of order $2^n$ with $n\ge 4$) $\times\; C_2$.
\end{enumerate}
\end{conjecture}

In \cite{Kah13} this conjecture was proved for $G$ an elementary Abelian $p$-group by giving an explicit description of $K(G,C_p)$. The simplest example shows an intricate network of subgroups behind the relative Brauer relations.  

\begin{proposition}[Kahn \cite{Kah13}]\label{kahn} $K(C_2\times C_2,C_2)$ has a basis with four elements $e_1-e_3-e_5-e_7+2e_{12}$, $e_3-e_{12}-e_{13}-e_4+e_{14}+e_{15}$, $e_5-e_{12}-e_{14}-e_6+e_{13}+e_{15}$, $e_7-e_{12}-e_{15}-e_8+e_{13}+e_{14}$ where $e$'s label distinct subgroups of $C_2\times C_2\times C_2$.
\end{proposition}

In this paper we prove Conjecture \ref{conj} for $G$ any finite Abelian $p$-group by giving the following description of $K(G,C_p)$:

\begin{theorem} \label{mainbb}
Let $p$ be a prime and $G$ a finite Abelian $p$-group. The $\mathbb Z$-module $K(G,C_p)$ of relative Brauer $(G,C_p)$-relations is generated by elements of the form $\text{\rm Ind}^{G\times C_p}_{K}\;\text{\rm  Inf}^{K}_P(\Theta' _{P})$ where $P=K/N\approx C_p\times C_p\times C_p$ are sub-quotients of $G\times C_p$ and $\Theta'_P$ are elements of $K(P)$, the module of Brauer $P$-relations.
\end{theorem}

\section{Rank lemmas and their proofs}

In this section we fix $G$ to be a finite Abelian $p$-group and endow the cyclotomic field $\mathbb Q(\zeta)$  with $\zeta$ a fixed primitive $p$-root of unity adjoined by the $(G,C_p)$-bimodule structure where $C_p$-action is given by multiplication by $\zeta$ and the $G$-action is {\em trivial}. 

Notice that a $(G,C_p)$-bimodule $W$ over the rationals is a right $\mathbb Q[C_p]$-\emph{free} module if and only if the module 
\begin{equation}\label{forget}
\empty_0W={\rm Res}_{1\times C_p} W
 \end{equation}
with the $G$-action forgotten is a left $\mathbb Q[C_p]$-free module.

\begin{lemma}\label{one}
$1_{G\times C_p}\oplus\mathbb Q(\zeta)$ represents an element in $R(G,C_p)$.
\end{lemma}
\begin{proof}
We have the following $\mathbb Q[C_p]$-isomorphism
$$
\empty_0\left(1_{G\times C_p}\oplus \mathbb Q(\zeta)\right)=1_{C_p}\oplus \mathbb Q(\zeta)\approx \mathbb Q\times \mathbb Q(\zeta)=\mathbb Q[C_p].
$$
This concludes the proof. \end{proof}

\begin{lemma}\label{rranks}
$R(G\times C_p)=R(G,C_p)\oplus \mathbb Z\cdot 1_{G\times C_p}$.
\end{lemma}

\begin{proof}
Since $G$ is a finite Abelian $p$-group, the group ring $\mathbb Q[G\times C_p]$ is a product of cyclotomic field extensions of $\mathbb Q$ obtained by adjoining primitive $p^\nu$-roots of unity $\xi_\nu$ where $\nu\ge 0$ are integers.  In particular, any irreducible $\mathbb Q[G\times C_p]$-module $W$ is a cyclotomic field, say $W=\mathbb Q(\xi_\nu)$, whose degree over $\mathbb Q$ is $d_\nu=p^{\nu-1}(p-1)$ and whose degree over $\mathbb Q(\zeta)$ is $p^{\nu-1}$. 

More precisely, there is a group homomorphism $\chi:G\times C_p\to \mathbb Q(\xi_\nu)^\times$ into the multiplicative group of the field $W=\mathbb Q(\xi_\nu)$ such that the elements $y\in G\times C_p$ act on $W$ by multiplication by $\chi(y)$. If we fix a generator $y_0$ of $C_p$, then $C_p$ acts on $W$ by multiplication by $\chi(y_0)$. Since $y_0^p=1$, we know that $\chi(y_0)$ is a $p$-root of unity. In particular, we distinguish two cases. 

If $\chi(y_0)=1$, we have the $\mathbb Q[C_p]$-isomorphisms  
$$\empty_0\left(W\oplus d_\nu\mathbb Q(\zeta)\right)=(\empty_0W)\oplus d_\nu\mathbb Q(\zeta)\approx d_\nu 1_{C_p}\oplus d_\nu \mathbb Q(\zeta)\approx d_\nu \mathbb Q[C_p].$$

If $\chi(y_0)=\zeta$ is a primitive $p$-root of unity,  we have the $\mathbb Q[C_p]$-isomorphisms 
$$
\empty_0\left(W\oplus p^{\nu-1}1_{G\times C_p}\right)\approx p^{\nu-1}\mathbb Q(\zeta)\oplus p^{\nu-1}1_{C_p}\approx p^{\nu-1}\mathbb Q[C_p].
$$
By Lemma \ref{one}, we deduce that for any irreducible $\mathbb Q[G\times C_p]$-module $W$, either $W-d_\nu1_{G\times C_p}$ or $W+p^{\nu-1}1_{G\times C_p}$ represents an element in $R(G,C_p)$. \end{proof}

\begin{lemma}[Goursat \cite{Lang02}]\label{goursat}
The subgroups of a direct product of two finite groups $\Gamma\times \Omega$ are in bijection with the quintuples $(K,N,A,B,\theta)$ of subgroups $N\trianglelefteq K\le \Gamma$ and $B\trianglelefteq A\le \Omega$ and isomorphisms $\theta:K/N\approx A/B$. 
\end{lemma}

The correspondence in Goursat Lemma is given by the following map
\begin{align*}
(K,N,A,B,\theta)\mapsto S=\{(kn,\theta(kN)b)|n\in N,\; k\in K,\; b\in B\}\le \Gamma\times\Omega
\end{align*}

Let $\mathcal S_G$ be the graph with a vertex $K$ for each subgroup $K\le G$ and an edge $(K,N)$  for each pair of subgroups $N\le K$ having index $[K:N]=p$.

\begin{lemma}\label{branks}
$\text{\rm rank }B(G\times C_p)-\text{\rm rank }B(G,C_p)=\text{\rm rank }B(G).$
\end{lemma}

\begin{proof}  By Proposition \ref{bbasis}, we have 
\begin{align*} \text{\rm rank }B(G)&=\#\text{ vertices in }\mathcal S_{G},\\
\text{\rm rank }B(G\times C_p)&=\#\text{ vertices in }\mathcal S_{G\times C_p}.\end{align*}
By Proposition \ref{rbbasis}, a basis for $B(G,C_p)$ is given by graphs of homomorphisms $\rho:K\to C_p$ from subgroups $K\le G$. Each such $\rho$ is either trivial $\rho=1$ or factors through a canonical map $K\to K/N$ and an automorphism of $C_p$ where $N\le K$ is a subgroup of index $[K:N]=p$. The number of automorphisms of $C_p$ is $p-1$. Hence, the number of graphs $K\times\rho$ equals the number of subgroups $K\le G$ ($\rho=1$) plus $(p-1)$ times the number of pairs $(K,N)$ with $N\le K$ having index $p$ ($\rho\neq 1$):
$$
\text{\rm rank }B(G,C_p)=\#\text{ vertices in }\mathcal S_G+(p-1)\cdot \#\text{ edges in }\mathcal S_G. 
$$
By Goursat's Lemma \ref{goursat}, the vertices of $\mathcal S_{G\times C_p}$ are in bijection with quintuples $(K,N,A,B,\theta)$ of subgroups $N\le K\le G$ and $B\le A\le C_p$ and isomorphisms $\theta:K/N\approx A/B$. We distinguish two cases: 

1) $A=B$, $K=N$, $\theta=1$ and 

2) $A=C_p$, $B=1$, $(K,N)$ is an edge in $\mathcal S_G$ and $\theta:K/N\to C_p$ is an isomorphism. 

Since $C_p$ has only two subgroups, we conclude that 
$$
\#\text{ vertices in }\mathcal S_{G\times C_p}=2\cdot \#\text{ vertices in }\mathcal S_G+(p-1)\cdot \#\text{ edges in }\mathcal S_G.
$$ 
The statement now follows by combining the formulas above. 
\end{proof}

\begin{proposition}\label{key}
$\text{\rm rank }K(G\times C_p)-\text{\rm rank }K(G,C_p)=\text{\rm rank }B(G)-1.$
\end{proposition}

\begin{proof}
From the diagram of exact sequences \eqref{exactseq}, we deduce the rank relations
\begin{align*}
\text{\rm rank }B(G\times C_p)&=\text{\rm rank }K(G\times C_p)+\text{\rm rank }R(G\times C_p) \\
\text{\rm rank }B(G, C_p)&=\text{\rm rank }K(G, C_p)+\text{\rm rank }R(G, C_p)
\end{align*}
By Lemma \ref{rranks},  $\text{\rm rank }R(G\times C_p)-\text{\rm rank }R(G,C_p)=1$ and thus if we take the difference between the two equations above and use Lemma \ref{branks}, we get the statement. 
\end{proof}

\section{Generators up to torsion}

For the rest of the paper we use the notation $\epsilon:G\to 1$ for the trivial map and the notation \eqref{grp} for the graph of a homomorphism. By Goursat Lemma \ref{goursat} the subgroups of $G\times C_p$ have the following structure 
\begin{align}\label{list1}
1\times \epsilon,\; L\times \epsilon,\; 1\times C_p,\; L\times C_p,\; L\times\lambda
\end{align}
where $\lambda:L\to C_p$ is a surjective homomorphism and $1\neq L\le G$. 

\begin{lemma}\label{list}
We have the following list of possible pairs of subgroups $(K,N)$ of $G\times C_p$ with $K/N\approx C_p\times C_p$:
\begin{align*}
&K=G'\times C_p,\; &N=L\times \rho &&\text{ with } \rho:L\to C_p,\; \rho(G'^p)=1,\; G'/L\approx C_p\\
&K=G'\times C_p,\; &N=L\times C_p &&\text{ with } G'/L\approx C_p\times C_p\\
&K=G'\times \lambda,\; &N=L\times \lambda &&\text{ with }  \lambda:G'\to C_p\text{ surjective},\; G'/L\approx C_p\times C_p\\
&K=G'\times \epsilon,\; &N=L\times \epsilon &&\text{ with } G'/L\approx C_p\times C_p
\end{align*}
where $(G',L)$ are pairs of subgroups of $G$ with $L<G'$. 
\end{lemma} 
\begin{proof}
The structure of the subgroup $K$ is given by \eqref{list1}. If $K$ is not a graph, the structure of a subgroup of $K$ is again give by \eqref{list1}.  If $K$ is a graph, any subgroup of $K$ is a subgraph. The constraint $K/N\approx C_p\times C_p$ translates to $G'/L\approx C_p\times C_p$ except for the case of a homomorphism $\rho:L\to C_p$ with $G'/L\approx C_p$. In this case, we always have $G'^p\le L$ as $G'/L\approx C_p$. If $\rho(G'^p)=1$, then an isomorphism $(G'\times C_p)/(L\times\rho)\approx C_p\times C_p$ is given by 
$$(yx_0^i,c)\mapsto(c_0^i,c\rho(y)^{-1}),\; \; y\in L,\; c\in C_p,\; i\in\mathbb Z$$ 
where $x_0\in G'$ is a generator of $G'/L\approx C_p$ and $c_0\in C_p$ is a generator of $C_p$. Indeed, each element of $G'$ is of the form $yx_0^i$ and the map is well defined since any other representation $y'x_0^j=yx_0^i$ gives $y'y^{-1}=x_0^{i-j}$ with $i-j=kp$ for some $k\in\mathbb Z$.  Hence, $$\rho(y'y^{-1})=\rho(x_0^{kp})=\rho(x_0^p)^k=1.$$ 
If $\rho(G'^p)\neq 1$ then $(G'\times C_p)/(L\times\rho)\approx C_{p^2}$ is generated by $(x_0,1)$:
$(x_0^p,1)\not\in L\times \rho$ as $\rho(x_0^p)\neq 1$, but $(x_0^{p^2},1)\in L\times\rho$  as $\rho(x_0^{p^2})=\rho(x_0^p)^p=1$ (recall that $x_0^p\in L$).
\end{proof}

Now we make a sublist $\mathcal L_{G\times C_p}$ of pairs $(K,N')$ selected from Lemma \ref{list} such that each $K$ appears \emph{exactly once} in $\mathcal L_{G\times C_p}$ as follows. For each non-trivial subgroup $G'\le G$ we \emph{choose} a subgroup $L'<G'$ such that $G'/L'\approx C_p\times C_p$ and if this is impossible, we \emph{choose} a subgroup $L'<G'$ such that $G'/L'\approx C_p$. 

\begin{definition} \label{Llist} With the choices above, the list $\mathcal L_{G\times C_p}$ of pairs $(K,N')$ of subgroups of $G\times C_p$ with $K/N'\approx C_p\times C_p$ is defined by 
\begin{align*}
&\text{ if } G'/L'\approx C_p\times C_p&&K=G'\times C_p,\; N'=L'\times C_p \\
&&&K=G'\times \lambda,\; N'=L'\times \lambda \text{ with } \lambda:G'\to C_p\text{ surjective}\\
&&&K=G'\times \epsilon,\; N'=L'\times \epsilon \\
&\text{ if } G'/L'\approx C_p&&K=G'\times C_p,\; N'=L'\times \epsilon 
\end{align*}
\end{definition}

\begin{lemma}\label{inspection}
The number of pairs $(K,N')$ of subgroups of $G\times C_p$ in the list $\mathcal L_{G\times C_p}$ with $K$ not a graph is one less than the number of subgroups of $G$. 
\end{lemma}

\begin{proof}
Each non-trivial subgroup $G'\le G$ falls  into one of the two categories of the Definition \ref{Llist}. Namely, $G'$ is non-cyclic if and only if admits a quotient $G'/L'\approx C_p\times C_p$. If $G'$ is cyclic, then it is non-trivial if and only if admits a quotient $G'/L'\approx C_p$. Since each product $K=G'\times C_p$ appears exactly once in the list $\mathcal L_{G\times C_p}$, this concludes the proof.
\end{proof}

\begin{lemma}\label{absolute}
 $\text{\rm rank }K(G\times C_p)=$ number of pairs $(K,N')$ in the list $\mathcal L_{G\times C_p}$.  
\end{lemma}

\begin{proof} By Corollary \ref{rank} the rank of $K(G\times C_p)$ equals the number of non-cyclic subgroups of $G\times C_p$. In the list $\mathcal L_{G\times C_p}$ the non-cyclic subgroups $K\le G\times C_p$ appear exactly once. Indeed, for a graph subgroup $G'\times\rho\le G\times C_p$ to admit a quotient $G'/L'\approx C_p\times C_p$ is equivalent with being non-cyclic. And a subgroup $G'\times C_p$ that admits a quotient $G'/L'\approx C_p$ is non-cyclic as a direct product. The two cases cover all the possibilities of non-cyclic subgroups without overlap. \end{proof}

By Proposition \ref{teta}, each pair $(K,N)$ from the Lemma \ref{list} produces an element $\text{Induf}(\Theta_{K/N})$ of $K(G\times C_p)$ which is defined as follows 
\begin{align*}
\text{Induf}:K(K/N)\to K(G\times C_p)\text{ is given by }S/N\mapsto S\text{ for }N\le S\le K\\
\Theta_{K/N}=(N/N)-\sum_{C'}(C'/L)+p(K/N)\text{ where }N\le C'\le K\text{ such that }C'/N\approx C_p
\end{align*}
Indeed, by \eqref{mac} and \eqref{inf} we have the following calculation 
\begin{align}\label{gen}
\text{\rm Ind}^{G\times C_p}_K\;\text{\rm  Inf}^{K}_{K/N}(\Theta_{K/N})=\text{Induf}(\Theta_{K/N})=N-\sum_{C'}C'+pK
\end{align}

\begin{theorem} \label{mainprop}
Let $p$ be a prime and $G$ a finite Abelian $p$-group. Then $K(G,C_p)[\tfrac{1}{p}]$  is a free $\mathbb Z[\tfrac{1}{p}]$-module whose basis is given by the elements  $\text{\rm Induf}(\Theta_{K/N'})$ indexed by the pairs of the form $(K,N')=(G'\times\rho,L'\times\rho)$ in the list $\mathcal L_{G\times C_p}$ where $\rho:G'\to C_p$ is a homomorphism and $G'/L'\approx C_p\times C_p$ is a sub-quotient of $G$.
\end{theorem}

\begin{proof}
By Proposition \ref{key} we have  $\text{rank }K(G,C_p)=\text{ rank }K(G\times C_p)-\text{ rank }B(G)+1$ where $\text{rank }B(G)=$ the number of subgroups of $G$. By Lemmas \ref{inspection} and \ref{absolute} we deduce that  $\text{rank }K(G,C_p)=$ the number of pairs $(K,N')$ with $K$ a graph, which are listed in $\mathcal L_{G\times C_p}$.  Observe that $K$ is a graph if there is a homomorphism $\rho:G'\to C_p$ such that $G'\le G$ and $K=G'\times \rho$. In this situation, any subgroup of $K$ must be a subgraph of the form $L\times \rho\le K$ where $L\le G'$ and $\rho$ is restricted to $L$. Hence, 
\begin{align}\label{basis}\text{Induf}(\Theta_{K/N'})=L'\times\rho-\sum_{C'}C'\times\rho+p(G'\times \rho)\end{align}
where $L'<C'<G'$ such that $C'/L'\approx C_p$ according to \eqref{gen}. We deduce that each element \eqref{basis} belongs to $B(G,C_p)$. For $(K,N')=(G'\times\rho,L'\times\rho)$ in the list $\mathcal L_{G\times C_p}$ the number of these elements equals the rank of $K(G,C_p)$. Since their dominant terms $pK$ under inclusion form a sub-basis of $B(G\times C_p)[\tfrac{1}{p}]$, the statement follows. 
\end{proof}

\begin{corollary}[ \cite{Kah13}]
For $p$ be a prime and $G$ a cyclic $p$-group, $K(G,C_p)=0$. 
\end{corollary}

\begin{proof}
Since $G$ is cyclic, $G$ does not admit sub-quotients of the form $G'/L'\approx C_p\times C_p$.
By Theorem \ref{mainprop}, $\text{rank }K(G,C_p)=0$. Recall that $B(G\times C_p)$ is a free $\mathbb Z$-module and thus, $K(G,C_p)\subset B(G\times C_p)$ is a free $\mathbb Z$-submodule. Hence, $K(G,C_p)=0$. 
\end{proof}

\section{The reduction to Type 2 generators}

We denote by $K'(G,C_p)\subset K(B,C_p)$ the submodule generated by the elements of $K(G,C_p)$ that are 'indufted' from sub-quotients of $G\times C_p$ isomorphic to $C_p\times C_p\times C_p$. The Theorem \ref{mainbb} states that $K'(G,C_p)=K(G,C_p)$. Since $K(G,C_p)\subset K(G\times C_p)$, by Theorem \ref{brauer} we know that each element $x$ of $K(G,C_p)$ is a $\mathbb Z$-linear combination of elements of the form $\text{Induf}(\Theta_{K/N})$ where $\Theta_{K/N}$ are defined as in \eqref{gen} for each pair $(K,N)$ given by Lemma \ref{list}. A careful analysis of the elements  $\text{Induf}(\Theta_{K/N})$ reveals the following classification:

{\bf Type 1.} For each pair of subgroups $L<G'<G$ with $G'/L\approx C_p\times C_p$ and each homomorphism $\alpha:G'\to C_p$ we define
$$
A_{G',L,\alpha}=L\times\alpha-\sum_{L<C'<G'} C'\times \alpha+pG'\times \alpha
$$
Here $C'$ runs over the subgroups $L<C'<G'$ with $C'/L\approx C_p$.

{\bf Type 2.} For each pair of subgroups $C<G'<G$ with $G'/C\approx C_p$ and each homomorphism $\beta:C\to C_p$ with $\beta(G'^p)=1$ we define
$$
B_{G',C,\beta}=C\times\beta-\sum_{\tilde{\beta}|C=\beta}G'\times\tilde{\beta}-C\times C_p+pG'\times C_p
$$
Here $\tilde{\beta}$ runs over the homomorphisms $\tilde{\beta}:G'\to C_p$ with $\tilde{\beta}|C=\beta$.

{\bf Type 3.} For each pair of subgroups $L<G'<G$ with $G'/L\approx C_p\times C_p$ we define
$$
D_{G',L}=L\times C_p-\sum_{L<C'<G'} C'\times C_p+pG'\times C_p
$$
Here $C'$ runs over the subgroups $L<C'<G'$ with $C'/L\approx C_p$. 

\begin{lemma}\label{type1}
The {\rm Type 1} elements $A_{G',L,\alpha}$ belong to $K'(G,C_p)$.
\end{lemma}
\begin{proof} For each sub-quotient $G'/L\approx C_p\times C_p$ of $G$ and homomorphism $\alpha:G'\to C_p$ we have the following isomorphism
 $$P=(G'\times C_p)/(L\times\alpha)\approx C_p\times C_p\times C_p,
 $$ 
which comes from $\varphi:G'\to G'/L\approx C_p\times C_p$ by sending $(x,c)\in G'\times C_p$ to the element $(\varphi(x),c\rho(x)^{-1})\in C_p\times C_p\times C_p.$
In this context, the element $A_{G',L,\alpha}$ is of the form $A_{G',L,\alpha}=\text{Induf}(\Theta'_P)\in K(G,C_p)$ for some $\Theta'_P\in K(P)$.
\end{proof}

\begin{corollary}\label{torsion}
For each $x\in K(G,C_p)$, either $x$ or $px$ belongs to $K'(G,C_p)$.
\end{corollary} 

\begin{proof} By Theorem \ref{mainprop} and its proof we know that for each $x\in K(G,C_p)$ either $x$ or $px$ is a $\mathbb Z$-linear combination of Type 1 elements and we apply Lemma \ref{type1}.
\end{proof}

\begin{lemma}\label{type3} For each pair $L<G'<G$ with $G'/L\approx C_p\times C_p$ we have
\begin{align}\label{eqtype3} (p+1)D_{G',L}\equiv \sum_{L<C'<G'} B_{G',C',\epsilon}-\sum_{L<C'<G'} B_{C',L,\epsilon}\mod K'(G,C_p)\end{align}
Here $C'$ runs over the subgroups $L<C'<G'$ with $C'/L\approx C_p$.
\end{lemma}

\begin{proof}
Recall that $B(G,C_p)$ is generated by graph-subgroups $K\times \rho <G\times C_p$ as in the Proposition \ref{rbbasis}. In this context, for each of the $p+1$ subgroups $L<C'<G'$ (see the next section) with $C'/L\approx C_p$, we have 
\begin{align*}
B_{G',C',\epsilon}&\equiv&-C'\times C_p+pG'\times C_p\mod B(G,C_p)\\
-B_{C',L,\epsilon}&\equiv&L\times C_p-pC'\times C_p\mod B(G,C_p)\\
D_{G',L} &\equiv&L\times C_p-\sum_{L<C'<G'} C'\times C_p+pG'\times C_p
\end{align*}
Hence, if we apply the operator $\sum_{L<C'<G'}$ to the first two equations we get
\begin{align*}
\sum_{L<C'<G'} B_{G',C',\epsilon}&\equiv&-\sum_{L<C'<G'} C'\times C_p+(p+1)pG'\times C_p\mod B(G,C_p)\\
-\sum_{L<C'<G'} B_{C',L,\epsilon}&\equiv&(p+1)L\times C_p-p\sum_{L<C'<G'} C'\times C_p\mod B(G,C_p)
\end{align*}
By definitions, $K(G.C_p)=B(G,C_p)\cap K(G\times C_p)$ where $K(G\times C_p)$ contains the Type 2 and Type 3 elements. Hence, by adding the last two equations, we get the relation \eqref{eqtype3} mod $K(G,C_p)$ where all the terms are 'indufted' from the sub-quotient $(G'\times C_p)/(L\times\epsilon)\approx C_p\times C_p\times C_p$. This proves \eqref{eqtype3} mod $K'(G,C_p)$. 
\end{proof}

Now we can reduce the proof of Theorem \ref{mainbb} to elements of $K(G,C_p)$ that are $\mathbb Z$-linear combinations  of Type 2 generators. Here is the precise statement 

\begin{proposition}\label{reduction} 
Each $x\in K(G,C_p)$ is a $\mathbb Z$-linear combination $\mod K'(G,C_p)$ of {\rm Type 2}-elements of the form $B_{G',C,\epsilon}$ with $G'/C\approx C_p$. 
\end{proposition}

\begin{proof}
Each element $x\in K(G,C_p)$ is a $\mathbb Z$-linear combination of the form
$$
x=\text{ Type 1 combination }+\text{ Type  2 combination } + \text{ Type 3 combination }
$$
By Lemma \ref{type3}, we have the following reduction $$(p+1)(\text{Type 3 combination})\equiv \text{ Type 2 combination }\mod K'(G,C_p)$$ 
By Corollary \ref{torsion}, we have $px\equiv 0\mod K'(G,C_p)$. By putting together the equations above and Lemma \ref{type1}, we get
$$
x=(p+1)x-px\equiv  \text{ Type 2 combination }\mod K'(G,C_p)
$$
If $G'<G$ is cyclic, then $C=G'^p<G'$ is the unique subgroup of index $p$.  In this case, $B_{G',C,\beta}=B_{G',C,\epsilon}$. 
If $G'<G$ is non-cyclic and $\beta:C\to C_p$ with $C<G'$ is such that $G'/C\approx C_p$ and $\beta\neq \epsilon$, then the difference
\begin{align*}
B_{G',C,\beta}-B_{G',C,\epsilon}= C\times\beta-\sum_{\tilde{\beta}|C=\beta}G'\times\tilde{\beta} - C\times\epsilon+\sum_{\tilde{\epsilon}|C=\epsilon}G'\times\tilde{\epsilon}
\end{align*}
is 'indufted' from $(G'\times C_p)/(L\times \beta)\approx  C_p\times C_p\times C_p$ if we take $L=\ker\beta<C$. This shows that the difference belongs to $K'(G,C_p)$ and it concludes the proof.
\end{proof}

Now we are ready to prove Theorem \ref{mainbb}. To that end, let $x\in K(G,C_p)$ be given. By Proposition \ref{reduction} we can represent $x$ by a Type 2 combination mod $K'(G,C_p)$. In what follows we will show how to eliminate all the Type 2 elements from that combination, concluding that $x\in K'(G,C_p)$. This proves Theorem \ref{mainbb}.

\section{The elimination algorithm}

The Type 2 elements $B_{G',L,\epsilon}$ generate a $\mathbb Z$-submodule $\mathcal M\subset K(G\times C_p)$ and each such generator is uniquely determined by a pair of subgroups $L<G'$ with $G'/L\approx C_p$. Hence, we can drop the $\epsilon$ from the notation $B_{G'L}=B_{G',L,\epsilon}$. Moreover, its image mod $B(G,C_p)$ is given by the formula
\begin{align}\label{signature}
B_{G'L}\equiv -L\times C_p+pG'\times C_p\mod B(G,C_p) 
\end{align}

\begin{definition}
The \emph{signature} homomorphism $\sigma:B(G\times C_p)\to B(G)$ is defined by sending  $L\times C_p\mapsto L$  for $L<G$ and any other basis elements to zero.
\end{definition}

For example, we say that the Type 2 generator $B_{G'L}$ has the signature $-L+pG$. 

\begin{lemma} The signature homomorphism $\sigma:B(G\times C_p)\to B(G)$ is surjective and its kernel is $B(G,C_p)$.
\end{lemma}

\begin{proof} Observe that $\sigma$ has a well defined section $\ell:B(G)\to B(G,C_p)$ sending $L\mapsto L\times C_p$ for $L<G$. Since the identity map  $\sigma\circ\ell:B(G)\to B(G)$ is surjective, so is $\sigma$. Moreover, by Proposition \ref{rbbasis} the $B(C\times C_p)$ is the direct sum of its submodules $B(G,C_p)$ and $\ell B(G)$. This shows that $\ker \sigma = B(G,C_p)$. 
\end{proof}

\begin{corollary}\label{injective}
The kernel of the restriction $\sigma:\mathcal M\to B(G)$ is $\mathcal M\cap K(G,C_p)$.
\end{corollary}

\begin{proof}
By the previous lemma, the kernel of the restriction $\sigma|\mathcal M$ is $\mathcal M\cap B(G,C_p)$. By combining this fact with $\mathcal M\subset K(G\times C_p)$ and $K(G,C_p)=K(G\times C_p)\cap B(G,C_p)$ we get the statement $\ker \sigma|\mathcal M=\mathcal M\cap K(G,C_p)$. 
\end{proof}

\begin{definition}
We call a \emph{resolution} starting at a subgroup $L$ and ending at a subgroup $G'$ any chain of intermediate subgroups $L=G_q<G_{q-1}<...<G_1<G_0=G'$ such that each subgroup $G_i$ has index $p$ in the next subgroup $G_{i+1}$ of the chain. 
\end{definition}

Here are a couple of basic facts \cite{Red89}. Between any two comparable subgroups $L<G'$ of a finite $p$-group there is at least one resolution starting at $L$ and ending at $G'$. If $L$ has index $p^2$ in $G'$, then $G'/L\approx C_{p^2}$ if from $L$ to $G'$ there is only one resolution and $G'/L\approx C_p\times C_p$ if there are at least two resolutions. In the latter case, there will be exactly $p+1$ such resolutions. 

\begin{lemma}\label{telescope}
Given any resolution $L=G_e<G_{e-1}<...<G_1<G_0=G'$ starting at a subgroup  $L$ and ending at a subgroup $G'$ of the group $G$ we have 
$$
\sigma\left(B_{G_{e-1}G_e}+pB_{G_{e-2}G_{e-1}}+...+p^{e-1}B_{G_0G_1}\right)=-G_e+p^eG_0
$$
\end{lemma}

\begin{proof}
Notice that $\sigma\left(B_{G_iG_{i+1}}\right)=-G_{i+1}+pG_i$ and apply a telescopic sum. 
\end{proof}

\begin{lemma}\label{zero}
Given any two resolutions starting at a subgroup $L$ and ending at a subgroup $G'$ of $G$, say
$$
L=G_e<G_{e-1}<...<G_1<G_0=G'\text{ and } L=H_e<H_{e-1}<...<H_1<H_0=G'
$$
we have the following relation
$$
\sum_{j=1}^e p^{e-j}B_{G_{j-1}G_j}\equiv \sum_{j=1}^e p^{e-j}B_{H_{j-1}H_j} \mod K'(G,C_p)
$$
\end{lemma}

\begin{proof}
By \cite{Red89} there is a sequence of resolutions $$L=G_e^{(i)}<G^{(i)}_{e-1}<...<G^{(i)}_1<G_0^{(i)}=G$$ for $i=0,1,2,...,n$ such that (1) for each $k$ we have $G_k^{(0)}=G_k$, $G_k^{(n)}=H_k$, and (2) for each $i$ there is $\lambda$ depending on $i$ with $G_\lambda^{(i+1)}\neq G_\lambda^{(i)}$ and $G_k^{(i+1)}= G_k^{(i)}$ if $k\neq \lambda$. 

In this context, notice that the following combinations belong to $K'(G,C_p)$  
\begin{align}\label{ted}
&\sum_{j=1}^e p^{e-j}B_{G^{(i+1)}_{j-1}G^{(i+1)}_j}- \sum_{j=1}^e p^{e-j}B_{G^{(i)}_{j-1}G^{(i)}_j} \notag \\
&=p^{e-\lambda-1}\left(pB_{G^{(i+1)}_{\lambda-1}G^{(i+1)}_\lambda}-pB_{G^{(i)}_{\lambda-1}G^{(i)}_\lambda}+B_{G^{(i+1)}_{\lambda}G^{(i+1)}_{\lambda+1}}-B_{G^{(i)}_{\lambda}G^{(i)}_{\lambda+1}}\right)
\end{align}
since the terms on the right hand side of the equation are associated with two resolutions starting at $G_{\lambda+1}^{(i+1)}=G_{\lambda+1}^{(i)}$ and ending at $G_{\lambda-1}^{(i+1)}=G_{\lambda-1}^{(i)}$ and thus, they are 'indufted' from $\left(G_{\lambda-1}^{(i)}\times C_p\right)/\left(G_{\lambda+1}^{(i)}\times \epsilon\right)\approx C_p\times C_p\times C_p$ as noted in basic facts. By adding up all the relations \eqref{ted} for $i=0,1,2,...,n$ we get get the result.
\end{proof}

Let the order of $G$ be $p^n$ and for each $k=0,1,2,...,n$ define $\mathcal G_k$ to be the set of all subgroups of index $p^k$ in $G$. According to the formula \eqref{signature}, the Type 2 elements are in bijection with their signatures as listed for each pair $(X_i,X_{i+1})$ with $X_{i+1}<X_i$ and $X_k\in\mathcal G_k$ in the table below

\begin{table}[h]
\caption{} 
\label{1}
\begin{center}
\begin{tabular}{c|c|c|c|c}
$-X_n+pX_{n-1}$&$-X_{n-1}+pX_{n-2}$&...&$-X_2+pX_1$&$-X_1+pX_0$
\end{tabular} 
\end{center}
\end{table}

Recall that the Type 2 elements generate a submodule $\mathcal M\subset K(G\times C_p)$.  Using Table \ref{1} as a mirror, we build a new system of generators for $\mathcal M$ using elementary operations. Namely, by basic facts, each pair $X_{i+1}<X_i$ can be extended to a resolution
\begin{align}\label{resx}
X_{i+1}<X_i<X_{i-1}<...<X_1<X_0=G
\end{align}
and using this resolution we replace $B_{X_iX_{i+1}}$ by 
\begin{align}\label{genx}
B_{X_i X_{i+1}}+pB_{X_{i-1}X_i}+...+p^i B_{X_0X_1}
\end{align}
By Lemma \ref{telescope} the signature table of the new system of generators is given by 

\begin{table}[h]
\caption{} 
\label{2}
\begin{center}
\begin{tabular}{c|c|c|c|c}
$-X_n+p^nX_{0}$&$-X_{n-1}+p^{n-1}X_{0}$&...&$-X_2+p^2X_0$&$-X_1+pX_0$
\end{tabular} 
\end{center}
\end{table} 

In this new table, the signatures appear with repetitions. More precisely, any two resolutions starting at $X_{i+1}$ and ending at $X_0$, say  resolution \eqref{genx} and resolution
\begin{align}\label{resy}
X_{i+1}<Y_i<Y_{i-1}<...<Y_1<X_0=G
\end{align}
produce the signature $-X_{i+1}+p^{i+1}X_0$ corresponding to generator \eqref{genx} and generator
\begin{align}\label{geny}
B_{Y_i X_{i+1}}+pB_{Y_{i-1}Y_i}+...+p^i B_{X_0Y_1}
\end{align}
Using the generator \eqref{genx} as a pivot and subtracting that generator from the generator \eqref{geny}, we can remove the signature duplicate in Table \ref{2}. According to  Lemma \ref{zero}, the zero represents an element in $K'(G,C_p)$. By using this procedure, we eliminate all the repetitions in Table \ref{2} mod $K'(G,C_p)$. Let $\mathcal S$ be the system thus obtained of generators for $\mathcal M$. By Proposition \ref{reduction}, given an element $x\in K(G,C_p)$, we can write it as a $\mathbb Z$-linear combination $y\in\mathcal M$ of elements in $\mathcal S$ mod $K'(G,C_p)$ as follows
$$
x\equiv y \mod K'(G,C_p) \text{ where }\sigma(y)=\sum_{i=1}^{n}\sum_{X\in\mathcal G_i}m_X(-X+p^iX_0)
$$
where $m_X\in\mathbb Z$ are the coefficients mod $K'(G,C_p)$ of the combination $y$. By Corollary \ref{injective} we must have $\sigma(y)=0$. Since the collection of subgroups $\cup_{i=1}^n\mathcal G_i$ is a sub-basis for $B(G)$, we deduce that $m_X=0$ for each $X$. This proves that $y\in K'(G,C_p)$ and thus, $x\in K'(G,C_p)$ proving Theorem \ref{mainbb}.

\section{An example}

By Theorem \ref{mainprop}, a basis for $K(C_p\times C_p,C_p)[\frac{1}{p}]$ is given by 
\begin{align*}
(1\times 1)\times \epsilon-\sum_C(C\times \epsilon)+p(C_p\times C_p)\times \epsilon\\
(1\times 1)\times \lambda-\sum_C(C\times \lambda)+p(C_p\times C_p)\times \lambda
\end{align*}
Here $C$ runs over the cyclic subgroups of order $p$ of $C_p\times C_p$ and $\lambda$ over the surjective homomorphisms $\lambda:C_p\times C_p\to C_p$. By direct counting, we deduce that $$\text{\rm rank }K(C_p\times C_p,C_p)=1+(p+1)(p-1)=p^2.$$
In particular, for $p=2$ we have $\text{rank }K(C_2\times C_2,C_2)=4$. Specifically, the lattice of subgroups for $ e_{16}=(C_2\times C_2)\times C_2$ is given by
\begin{align*}{ e_9=(1\times C_2)\times C_2,\; e_{10}=(C_2\times 1)\times C_2,\; e_{11}=\Delta\times C_2}\\
e_{12}=(C_2\times C_2)\times \epsilon,\; e_{13}=(C_2\times C_2)\times p_1,\; 
e_{14}=(C_2\times C_2)\times p_2\\ e_{15}=(C_2\times C_2)\times \sigma\\
{ e_2=(1\times 1)\times C_2},\; e_3=(1\times C_2)\times\epsilon,\; e_4=(1\times C_2)\times p_2\\
e_5=(C_2\times 1)\times \epsilon,\; e_6=(C_2\times 1)\times p_1,\; e_7=\Delta\times\epsilon,\; e_8=\Delta\times\delta\\
e_1=(1\times 1)\times \epsilon
\end{align*}
where $\Delta\subset C_2\times C_2$ is the diagonal subgroup, $\rho:C_2\times C_2\to (C_2\times C_2)/\Delta$ is the canonical projection, $\delta:\Delta\to C_2$ is the unique isomorphism, $p_1, p_2:C_2\times C_2\to C_2$ are the projections on the first and the second component respectively.

The generators of $K(C_2\times C_2\times C_2)$ are
\begin{align*}
&{ E_9}=e_1-{e_2}-e_3-e_4+2{e_9}
&&{ E_{10}}=e_1-{e_2}-e_5-e_6+2{ e_{10}}\\
&{ E_{11}}=e_1-{e_2}-e_7-e_8+2{ e_{11}}
&&E_{12}=e_1-e_3-e_5-e_7+2e_{12}\\
&E_{13}=e_1-e_3-e_6-e_8+2e_{13}
&&E_{14}=e_1-e_4-e_5-e_8+2e_{14}\\
&E_{15}=e_1-e_4-e_6-e_7+2e_{15}
&&{ E_2}={ e_2}-{ e_9}-{ e_{10}}-{ e_{11}}+2{ e_{16}}\\
&{ E_3}=e_3-{ e_9}-e_{12}-e_{13}+2{ e_{16}}
&&{ E_4}=e_4-{ e_9}-e_{14}-e_{15}+2{ e_{16}}\\
&{ E_5}=e_5-{ e_{10}}-e_{12}-e_{14}+2{ e_{16}}
&&{ E_6}=e_6-{ e_{10}}-e_{13}-e_{15}+2{ e_{16}}\\
&{ E_7}=e_7-{ e_{11}}-e_{12}-e_{15}+2{ e_{16}}
&&{ E_8}=e_8-{ e_{11}}-e_{13}-e_{14}+2{ e_{16}}
\end{align*}

As in Proposition \ref{kahn}, we can prove that a basis for $K(C_2\times C_2,C_2)$ is given by
$$E_{15}, \; E_4-E_3, \; E_6-E_5, \; E_8-E_7.$$


\end{document}